\documentclass[letterpaper,11pt]{article}
\usepackage{amsmath,amssymb,amsthm,amscd,latexsym,epsfig,subfig,setspace,ctable}
\newcommand{\be}{\begin{equation}}
\newcommand{\ee}{\end{equation}}

\newenvironment{pf}{\textbf{Proof:}\quad}
{\nopagebreak{\begin{flushright}{$\blacksquare$}\end{flushright}}}

\renewcommand{\O}{\mathcal{O}}

\newtheorem{thm}{Theorem}
\newtheorem{lem}{Lemma}

\newtheorem{remark}{Remark}

\numberwithin{equation}{section}
\numberwithin{lem}{section}

\newcommand{\e}{\epsilon}

\renewcommand{\O}{\mathcal{O}}

\begin{document}
\title{A proof of anomalous invasion speeds in a system of coupled Fisher-KPP equations}

\author{Matt Holzer\footnote{mholzer@gmu.edu} \\ Department of Mathematical Sciences \\ George Mason University \\ Fairfax, VA 22030 }

\date{ \today}

\maketitle
\begin{abstract}
This article is concerned with the rigorous validation of anomalous spreading speeds in a system of coupled Fisher-KPP equations of cooperative type. Anomalous spreading refers to a scenario wherein the coupling of two equations leads to faster spreading speeds in one of the components.  The existence of these spreading speeds can be predicted from the linearization about the unstable state.  We prove that initial data consisting of compactly supported perturbations of  Heaviside step functions spreads asymptotically with the anomalous speed.  The proof makes use of a comparison principle and the explicit construction of sub and super solutions.

\end{abstract}

%\noindent {\bf MSC numbers:}


\section{Introduction}

In this article we study spreading properties for the following system of coupled reaction-diffusion equations,
\begin{eqnarray}
u_t&=& du_{xx}+\alpha u(1-u)+\beta v \nonumber \\
v_t&=&  v_{xx}+v(1-v). \label{eq:main}
\end{eqnarray}
The parameters $d, \alpha$ and $\beta$ are positive.  The system was introduced in \cite{holzer-anomalous} as a prototypical example of a system of reaction-diffusion equations exhibiting anomalous spreading.  In the context of (\ref{eq:main}), anomalous spreading refers to a phenomena where the spreading speed for the $u$ component observed in the coupled regime ($\beta>0$) greatly exceeds that of the uncoupled case.  When the system is linearized about the unstable zero state these anomalous speeds are easily calculated to exist within certain regions in parameter space.  Our main goal is to prove that for some of these parameter values Heaviside step function initial data spreads with this anomalous speed in the nonlinear system (\ref{eq:main}) as well.

The equation governing both the $u$ and $v$ components in isolation (i.e. when $\beta=0$) in (\ref{eq:main}) is the Fisher-KPP equation, see \cite{fisher37,kolmogorov37}.  This equation has been studied for many decades and the dynamics are well understood.  Compactly supported, positive initial data converges to a nonlinear traveling front propagating  asymptotically with speed $2\sqrt{d\alpha}$.  This speed is also the spreading speed of solutions of the linearized equation $u_t=du_{xx}+\alpha u$.  For this reason, the Fisher-KPP equation is called {\em linearly determinate} since the nonlinear spreading speed is exactly the linear one and the Fisher-KPP front propagating at this speed is referred to as a {\em pulled front}.

It turns out that (\ref{eq:main}), like the Fisher-KPP equation, is a linearly determinate system in the sense that the nonlinear spreading speeds can be computed from the linearized equation. Spreading speeds for the linearized system can be computed systematically from the associated dispersion relation by locating pinched double roots, see section~\ref{sec:linear} and \cite{vansaarloos03,holzer-criteria} for more details.  The skew-product nature of the linearization of system (\ref{eq:main}) implies that the dispersion relation for the full system is the product of the dispersion relations of the reduced system.  Pinched double roots can be computed explicitly and  there are three different spreading speeds that might be of interest: the spreading speed of the $u$ component in isolation (the Fisher-KPP speed $2\sqrt{d\alpha}$), the spreading speed of the $v$ component in isolation (the Fisher-KPP speed $2$) and the spreading speed of the $u$ component induced by the coupling to the $v$ component.  Straightforward computations in section~\ref{sec:linear} show that the largest of these three speeds depends on the parameter values $(d,\alpha)$ and can be summarized in the following table (see also Figure~\ref{fig:parameterspace}).  One observes two regions in parameter space where the fastest linear spreading speed is given by the spreading speed of the $u$ component induced by the $v$ component.  We call this speed the {\em anomalous} spreading speed.

\begin{center} \begin{tabular}{| c |  c | c |} \hline
{parameter regime}  & \text{linear spreading speed} & \text{pinched double root} \\ \hline  
 $\mathrm{I}$ &   $2\sqrt{d\alpha}$ & $\nu_u^+=\nu_u^-$ \\ \hline
 $\mathrm{II}$ &  $2$ & $\nu_v^+=\nu_v^-$ \\ \hline
 $\mathrm{III}$ &   $s_{anom}$ & $\nu_v^+=\nu_u^-$ \\ \hline
 $\mathrm{IV}$ &   $s_{anom}$ & $\nu_u^+=\nu_v^-$ \\ \hline
\end{tabular}\end{center}

With the linear spreading speeds computed, a natural question is whether these speeds are observed in the nonlinear system.  This is more complicated and it is {\em not} the case that the observed speed in the nonlinear system is simply the fastest of these linear speeds.  This issue has been explored in depth in \cite{holzer-anomalous,holzerLV,holzer-criteria} where a distinction is made between linear spreading speeds based upon the analyticity (or lack thereof) of the pointwise Green's function in a neighborhood of the singularities enforcing these linear speeds.  In \cite{holzer-anomalous}, the double roots leading to anomalous spreading in the linear system were called {\em relevant} if anomalous spreading speeds were also observed numerically in the nonlinear regime and {\em irrelevant} if the invasion speed was slower than the anomalous one in the nonlinear system despite the existence of faster spreading speeds in the linear system.  It was proven in \cite{holzer-anomalous}, that for parameters in the irrelevant regime, $\mathrm{III}$, the observed spreading speed was the largest of the two spreading speeds in isolation, $\max\{2,2\sqrt{d\alpha}\}$.  

In this article, we will prove that the nonlinear spreading speed for parameters in the relevant regime $\mathrm{IV}$ is the anomalous speed $s_{anom}$.  This result confirms numerical observations in \cite{holzer-anomalous}. The primary analytical challenge is that anomalous spreading necessarily involves different components spreading at different speeds.  In this way, a traditional traveling wave analysis is incapable of describing this phenomena.  Instead, we rely on the existence of a comparison principle for (\ref{eq:main}) and will require some details concerning the behavior of the leading edge of the $v$ component.

%An interesting feature of (\ref{eq:main}) is that for some parameter values there exists multiple linear spreading speeds and the question then becomes which of these linear spreading speeds is observed in the nonlinear system.  

We review some related work.  There are several numerical studies, \cite{bell09,elliott12,holzer-anomalous}, that suggest that relevant double roots do lead to anomalous spreading speeds in the nonlinear regime.  Rigorous results in the relevant regime are, to our knowledge, few.   Freidlin \cite{freidlin91} proved that if (\ref{eq:main}) is modified by introducing a $\beta u$ term in the $v$ component then for any $\beta>0$ both components will spread with a faster speed for parameters including those in $\mathrm{III}$ and $\mathrm{IV}$.  As $\beta\to 0$, these speeds limit on the anomalous one and therefore this implies that nonlinear spreading speeds are not necessarily continuous functions of the system parameters.  The possibility of anomalous spreading speeds in partially coupled equations was first noted by Weinberger, Lewis and Li \cite{weinberger07} and bounds on the spreading speeds were obtained for example problems.  Recently, it was shown in \cite{holzer-criteria} that relevant double roots imply that any traveling front $(U(x-st),V(x-st))$, moving slower than the anomalous speed and having steep exponential decay is unstable due to the presence of an embedded resonance pole due to the relevant double root.  We also note that there is a large literature studying invasion phenomena in many different contexts.  We point the reader to \cite{vansaarloos03} for a review.  In particular, studies of coupled systems of reaction-diffusion equations have garnered significant interest.  Most related to current study are works on spreading speeds in general cooperative systems \cite{weinberger02,li05} as well as the specific example of coupled Fisher-KPP equations \cite{raugel98,ghazaryan10,iida11}.

We now state our main result.  In brief, we will show that for $(d,\alpha)\in\mathrm{IV}$, the spreading speed of (\ref{eq:main}) is exactly the anomalous speed $s_{anom}$ predicted by the linearized system.  The precise result is as follows.

\begin{thm}\label{thm:main} Define the invasion point 
\[ \kappa(t)=\sup_{x\in\mathbb{R}}\left\{ x\ | \ u(t,x)\geq \frac{1}{2}\right\}.\] 
Define the {\em selected speed}, $s_{sel}$ by 
\[ s_{sel}=\lim_{t\to \infty}\frac{\kappa(t)}{t}.\]
Consider (\ref{eq:main}) with $(d,\alpha)\in\mathrm{IV}$ and $\beta>0$.  Fix initial data $0\leq u_0(x)\leq \frac{1}{2}+\frac{1}{2}\sqrt{1+\frac{4\beta}{\alpha}}$ and $0\leq v_0(x)\leq 1$, each consisting of a compactly supported perturbation of the Heaviside step function $\left(\frac{1}{2}+\frac{1}{2}\sqrt{1+\frac{4\beta}{\alpha}}\right)H(-x)$ and $H(-x)$, respectively.  Then the selected spreading speed of the $u$ component is the anomalous one, i.e. $s_{sel}=s_{anom}$. 
\end{thm}

The proof of Theorem~\ref{thm:main} involves the explicit construction of sub and super solutions and relies on the fact that (\ref{eq:main}) is cooperative and each component satisfies the comparison principle.  The sub-solution consists of a weakly decaying traveling front solution concatenated with a solution of the linearized equation.  Since it is the leading edge of the $v$ component that is driving the anomalous spreading in the $u$ component we will require some results regarding behavior of the solution in the leading edge.  We make use of a recent study \cite{hamel13} of the Fisher-KPP equation that shows that sub-solutions can be constructed from solutions of the linearized equation in a moving coordinate frame with Dirichlet boundary conditions placed at the left edge of the domain.  Using this sub-solution, we can find a wedge in space-time for which a pure exponential solution is a sub-solution for the $v$ component and in turn leverage this to find a sub-solution for the $u$ component in this region.  It is interesting to note that the wedge for which the pure exponential is a sub-solution propagates at the group velocity associated to that exponential decay rate.  This observation was previously made by a formal analysis in \cite{booty93}.

The article is organized as follows.  In section~\ref{sec:prelims}, we outline some preliminaries necessary for the proof of Theorem~\ref{thm:main}.  This includes a review of linear spreading speeds, the construction of a sub-solution for the $v$ component and a subsequent construction of the sub-solution for the $u$ component.  In section~\ref{sec:proof}, we use these sub-solutions to prove Theorem~\ref{thm:main}.  Finally, in section~\ref{sec:discussion} we discuss some generalizations and directions for future work.

\section{Preliminaries}\label{sec:prelims}
This section establishes some preliminaries necessary for the proof of Theorem~\ref{thm:main}.  In section~\ref{sec:linear}, we review the notions of linear spreading speed, envelope and group velocities.  In section~\ref{sec:subV}, we review some facts concerning the Fisher-KPP equation and establish a sub-solution for the uncoupled $v$ component. Finally, in section~\ref{sec:subU} we use the analysis in the previous two sections to establish a compactly supported sub-solution for the $u$ component.
\subsection{The linearization about the zero state}\label{sec:linear}
In this section, we consider the linearization of (\ref{eq:main}) about the unstable homogeneous state $(u,v)=(0,0)$.  We compute and discuss the significance of the linear spreading speed, envelope velocities and group velocities.  Note that much of the material in this sub-section was covered in \cite{holzer-anomalous} and we refer the reader there for a more in depth treatment.  

We transform  (\ref{eq:main}) to a moving coordinate frame via $y=x-st$ for $s>0$ and linearize about the unstable homogeneous state $(u,v)=(0,0)$, 
\be \left(\begin{array}{c} u_t \\ v_t \end{array}\right) = \left(\begin{array}{cc}
d\partial_{yy}+s\partial_y+\alpha  & \beta \\
0 & \partial_{yy}+s\partial_y+1 \end{array}\right)\left(\begin{array}{c} u \\ v \end{array}\right).  \label{eq:mainlinear}\ee

The {\em linear spreading speed} associated to (\ref{eq:mainlinear}) is the asymptotic speed of propagation associated to compactly supported initial data.  We refer the reader to \cite{vansaarloos03} and \cite{holzer-criteria} for more detailed discussions of the linear spreading speed and the manner in which it is computed.  We sketch the details here.  Consider solutions of (\ref{eq:mainlinear}) of the form $(u,v)^T=e^{\nu y+\lambda t}(u_0,v_0)^T$, for $\nu,\lambda\in\mathbb{C}$.  These solutions exist for those values of $\nu$ and $\lambda$ for which the dispersion relation,
\be d_s(\nu,\lambda)=(d\nu^2+s\nu+\alpha-\lambda)(\nu^2+s\nu+1-\lambda),\label{eq:disp}\ee
is equal to zero. Let $d_s(\nu,\lambda)=d_{u}(\nu,\lambda)d_v(\nu,\lambda)$ for the two factors in this dispersion relation. The linear spreading speed for a system of parabolic equations can be computed from the double roots of this dispersion relation and is defined as 
\be s_{lin}=\sup_{s\in\mathbb{R}}\left\{ d_s \ \text{has a pinched double root} \ (\nu^*,\lambda^*)  \ \text{with} \ \mathrm{Re}(\lambda^*)>0\right\},\label{eq:PDRcriterion}\ee
where $(\nu^*,\lambda^*)$ is a pinched double root if
\[ d_s(\nu^*,\lambda^*)=0, \quad \partial_\nu d_s(\nu^*,\lambda^*)=0, \quad \mathrm{Re}(\nu^\pm(\lambda))\to \pm\infty \ \text{as} \ \mathrm{Re}(\lambda)\to\infty,\]
where $\nu^\pm(\lambda)\to\nu^*$ as $\lambda\to\lambda^*$. 

The linear system (\ref{eq:mainlinear}) is triangular and the linear spreading speed can be computed explicitly.  To accomplish this we first compute the roots of the dispersion relation (\ref{eq:disp}).  These are
\begin{eqnarray*}
\nu_u^\pm(s,\lambda)&=&-\frac{s}{2d}\pm\frac{1}{2d}\sqrt{s^2-4d\alpha+4d\lambda} \\
\nu_v^\pm(s,\lambda)&=&-\frac{s}{2}\pm\frac{1}{2}\sqrt{s^2-4+4\lambda}.
\end{eqnarray*}

Pinched double roots occur for those values of $s$ and $\lambda$ for which $\nu_{u,v}^+(s,\lambda)=\nu_{u,v}^-(s,\lambda)$.  Two pinched double roots are evident for $\lambda=0$ when $\nu_{u}^+(2\sqrt{d\alpha},0)=\nu_{u}^-(2\sqrt{d\alpha},0)$ or $\nu_{v}^+(2,0)=\nu_{v}^-(2,0)$.  Note that the speeds $s_u=2\sqrt{d\alpha}$ and $s_v=2$ are the linear spreading speeds of the $u$ and $v$ component in isolation.  

A third spreading speed is possible for the $u$ component whenever $\nu_{u}^\pm(s,0)=\nu_{v}^\mp(s,0)$.  We call this spreading speed the anomalous speed, with formula,
\[ s_{anom}= \sqrt{\frac{\alpha-1}{1-d}}+\sqrt{\frac{1-d}{\alpha-1}},\]
and can be found  by direct calculation.  
This spreading speed only occurs for a subset of parameters. The   $(d,\alpha)$ parameter space can be decomposed into four regions according to which of the three pinched double roots gives rise to the linear spreading speed.  These regions were referenced in the table above and are depicted in the left panel of Figure~\ref{fig:parameterspace}.  The four regions have the following descriptions, 
\begin{eqnarray*}
\mathrm{I}&=& \left\{(d,\alpha) \ | \  \alpha\geq\frac{d}{2d-1}, \ d>\frac{1}{2} \right\} \\
\mathrm{II}&=& \left\{(d,\alpha) \ | \  \alpha\leq 2-d \right\} \\
\mathrm{III}&=& \left\{(d,\alpha) \ | \  2-d<\alpha <\frac{d}{2d-1}, d>1,  \right\} \\
\mathrm{IV}&=&  \left\{(d,\alpha) \ | \ 2-d<\alpha \ (d\leq 1/2) \ ,   2-d<\alpha <\frac{d}{2d-1} \ (1/2<d<1) \right\}.
\end{eqnarray*}

A convenient way to understand the double roots that lead to anomalous spreading speeds is to graph the $s$ and $\nu$ values for which the dispersion relations for the $u$ and $v$ components in isolation satisfy $d_u(\nu,0)=0$ and $d_v(\nu,0)$.  These graphs depict the {\em envelope velocity} of a mode $e^{\nu x}$ for $\nu\in\mathbb{R}^-$, defined as the speed at which this exponential propagates in the linear system (\ref{eq:mainlinear}).  A short computation gives,
\[ s_{env,u}=-d\nu-\frac{\alpha }{\nu}, \quad s_{env,v}=-\nu-\frac{1}{\nu}.\]
The intersection of the curves of envelope velocities for the $u$ and $v$ component are double roots and are pinched if the derivatives of the two envelope velocity curves have opposite signs.  See Figure~\ref{fig:parameterspace}.

We finally mention a third quantity that will arise in our construction of sub-solutions.  This quantity is the {\em group velocity}, which is defined in general as $s_g=-\partial_{\nu}d / \partial_\lambda d (\nu^*,\lambda^*)$, where $(\nu^*,\lambda^*)$ is a root of the dispersion relation.  When $d=d_v$ this calculation gives that the group velocity of a mode $e^{\nu x}$ in the $v$ component is $s_g(\nu)=-2\nu$.  The group velocity of the $v$ component will surface in the construction of sub-solutions for the $v$ component later in this section.  We find that the region ahead of the front interface where the $v$ component has a sub-solution given by an exponential with decay rate $\nu$ is a small interval propagating with the group velocity $-2\nu$.  See \cite{booty93} for a formal study of the role of group velocities in the dynamics of the leading edge of the Fisher-KPP invasion front.

\begin{figure}[ht]
\centering
   \includegraphics[width=0.4\textwidth]{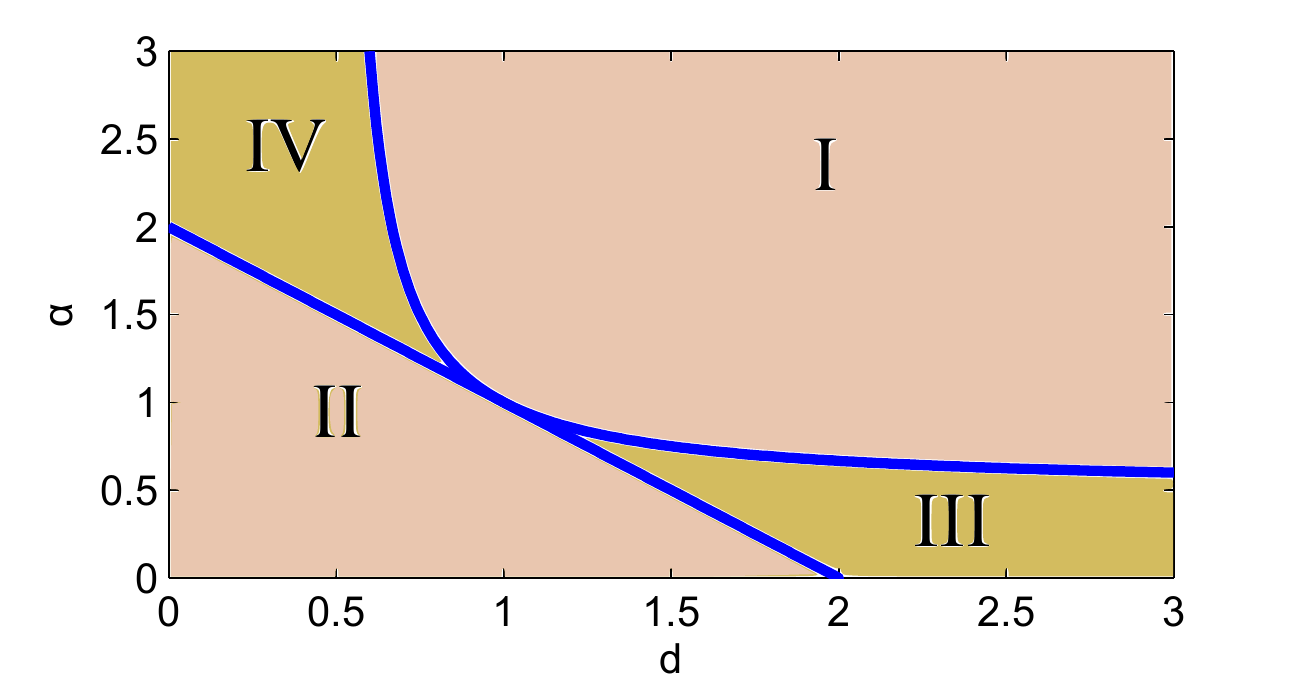}
   \includegraphics[width=0.4\textwidth]{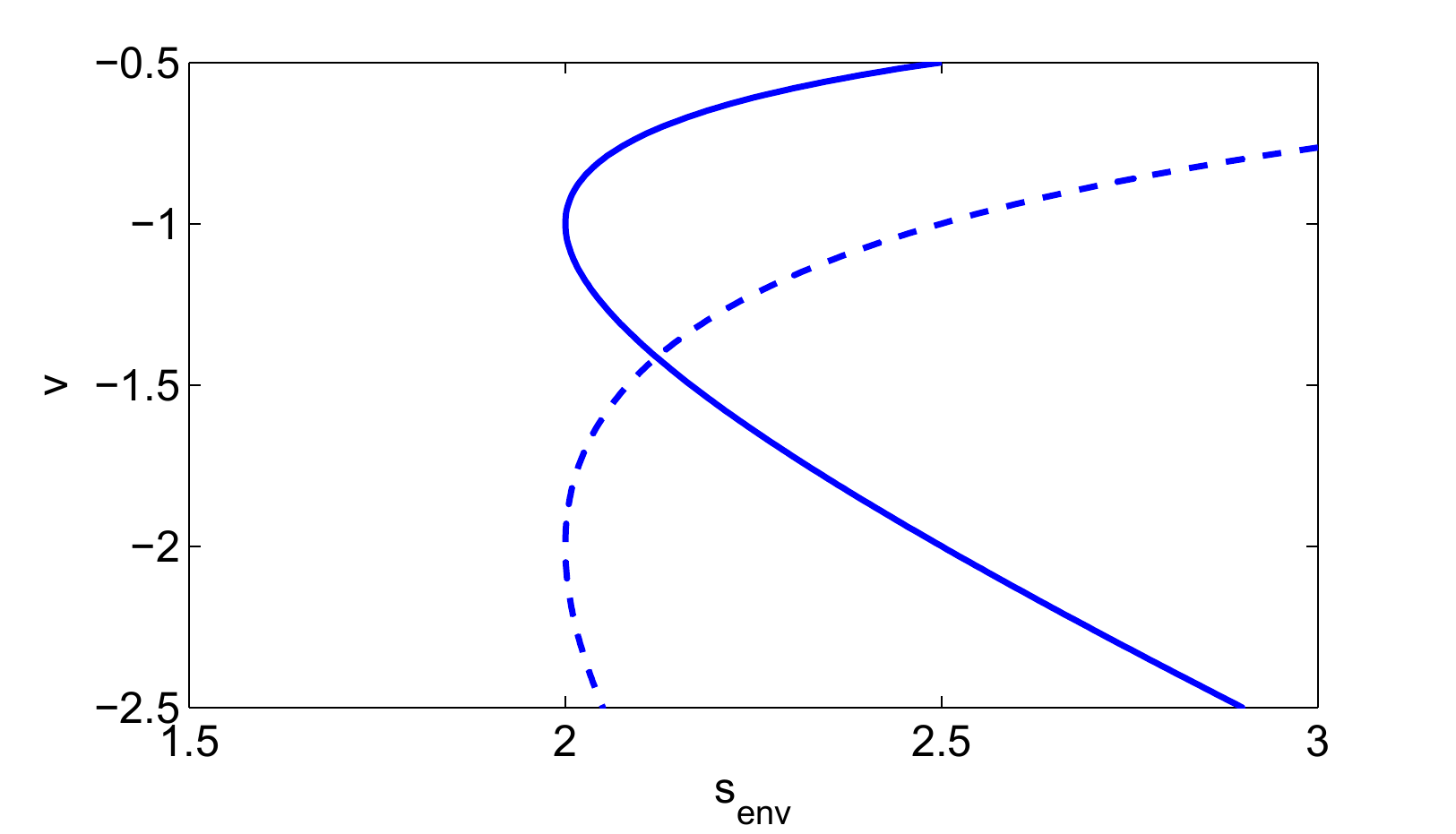}  
\caption{On the left is the $(d,\alpha)$ parameter space with regions $\mathrm{I}-\mathrm{IV}$ depicted.  Parameters in regions $\mathrm{III}$ and $\mathrm{IV}$ lead to anomalous spreading speeds in the linear system (\ref{eq:mainlinear}).  We focus on regime $\mathrm{IV}$ in this article and prove that these parameter values have nonlinear spreading speed given by the anomalous one.  On the right is the graph of the envelope velocities for representative values of $(d,\alpha)\in\mathrm{IV}$, here $d=0.5$ and $\alpha=2$.  The solid line represents the envelope velocities for $v$ and the dashed line represent the envelope velocities for $u$. The minimums of $s_{env}(\nu)$ correspond to the linear spreading speeds of $u$ and $v$ in isolation, while the point where the two curves cross gives the anomalous speed.}
\label{fig:parameterspace}
\end{figure}

\begin{remark}
The definition of the linear spreading speed using the double root criterion (\ref{eq:PDRcriterion}) is not always a faithful measure of the actual rates of spreading in a given linear system.  This can be observed in (\ref{eq:mainlinear}) where when $\beta=0$, the linear spreading speeds arising from pinched double roots involving one root from the $u$ component and one from the $v$ component are not relevant in the fully un-coupled case.  This fact was explored in depth in \cite{holzer-criteria}, where it was shown that a more accurate measure of the linear spreading speed is given by singularities of the pointwise Green's function.  Since we will always consider the case $\beta>0$, this ambiguity does not come into play here and we will use the definition given above.  We point the reader to \cite{holzer-anomalous} for details of the pointwise Green's function for the linearization (\ref{eq:mainlinear}).
\end{remark}

\subsection{A sub-solution of the Fisher-KPP equation}\label{sec:subV}
The evolution of the $v$ component is governed by the classical Fisher-KPP equation and its dynamics are well understood.  Initial data, $0\leq v(0,x)\leq 1$, a compactly supported perturbation of a Heaviside step function will evolve into a traveling front propagating with the unique, minimal KPP speed of two\footnote{with a logarithmic correction to the wavespeed, see \cite{bramson83}}.  We expect that for parameters in the anomalous regime $\mathrm{IV}$ the dynamics in the leading edge of the $v$ component is driving the anomalous spreading in the $u$ component.  We use the method of sub and super solutions, see for example \cite{fife77}, to prove that this is the case.  In this section, we construct a sub-solution for the $v$ component.

Of particular interest here is the behavior of the solution $v(t,x)$ far ahead of the front interface.  We will show that for any speed $\sigma>2$, the exponential function
\be \underline{v}(t,x)= e^{\nu_v^-(\sigma)(x-\sigma t)}e^{-\delta t},\label{eq:vsubsol}\ee
is a sub-solution for the $v$ component for $t>T^*$ on an interval $[\tau_-(t),\tau_+(t)]$, where $\tau_\pm(t)$ are functions of $\delta>0$ and $\sigma$.  We will show that the interval $[\tau_-(t),\tau_+(t)]$ is a wedge in space time whose midpoint propagates with speed $s_g(\nu_u^-(\sigma))$ in the asymptotic limit $t\to\infty$.  

We will do this with the aid of a secondary sub-solution introduced in \cite{hamel13}.
\begin{lem}(Proposition 3.1 of \cite{hamel13})\label{lem:dirichletsub}
Consider $\sigma>2$ and let $y=x-\sigma t$.  Let $q_0(y):\mathbb{R}^+\to[0,1]$ be compactly supported. There exists a function $G(t,y)$, with $|G(t,y)|<C$ for some $C(\sigma,q_0)>0$  such that 
\be q(t,y)=e^{(1-\frac{\sigma^2}{4})t}e^{-\frac{\sigma}{2}y}e^{-\frac{y^2}{4t}}G(t,y),\label{eq:dirsub}\ee
is a sub-solution of the $v$ component for $y\in[0,\infty)$. 
\end{lem} 

\begin{proof} See \cite{hamel13}, although we sketch the proof here for completeness.  We work in a coordinate frame moving to the right at a fixed speed $\sigma>0$.  Let $y=x-\sigma t$. 
A function $q(t,y)$ is a sub-solution if 
\[ N(q)=q_t-q_{yy}-sq_y-q+q^2\leq 0,\]
for all $(t,y)\in \mathbb{R}^+\times\mathbb{R}^+$. 
Consider the initial value problem for the linearized equation with a  Dirichlet boundary condition imposed at $y=0$,
\[ \tilde{q}_t=\tilde{q}_{yy}+\sigma \tilde{q}_y+\tilde{q}, \quad \tilde{q}(0,y)=\tilde{q}_0(y), \quad \tilde{q}(t,0)=0.\]
Asume that $\tilde{q}_0(y)$ is compactly supported.  This equation has the explicit solution,
\[ \tilde{q}(t,y)=e^{(1-\frac{\sigma^2}{4})t}e^{-\frac{\sigma}{2}y}\int_0^\infty \frac{\left(e^{-\frac{(y-y')^2}{4t}}-e^{-\frac{(y+y')^2}{4t}}\right)}{\sqrt{4\pi t}}\tilde{q}_0(y')dy'.\]
As $\tilde{q}(t,y)$ is a solution to the linearized equation it is not a sub-solution since $N(\tilde{q})=\tilde{q}^2\geq0$.  However, one can construct  a function $A(t)\geq0$ such that $q(t,y)=A(t)\tilde{q}(t,y)$ is a sub-solution.  To do this, note that $|\tilde{q}(t,y)|<Ce^{-\omega t}$ for some $\omega>0$ and some $C(q_0)$ independent of $y$.  Then if
\[ A'(t)=-Ce^{-\omega t}A^2,\]
we have that $q(t,y)$ is a sub-solution.  The solution of this differential equation can be calculated explicitly,
\[ A(t)=\frac{A_0\omega}{\omega+CA_0(1-e^{-\omega t})}.\]
Taking for example  $A_0=1$, we find that the solution is bounded away from zero as well as from above, i.e. there exists $A_1>0$ such that  $A_1<A(t)<1$.

We then have the existence of a sub-solution $q(t,y)$, which can be factored as in (\ref{eq:dirsub}) with
\[ G(t,y)=\frac{A(t)}{\sqrt{2\pi t}} \int_0^\infty \left(e^{-\frac{y'(y'-2y)}{4t}}-e^{-\frac{y'(2y+y')}{4t}}\right)\tilde{q}_0(y')dy',\]
from which we observe that $G(t,0)=0$ and $|G(t,y)|<C$ for some $C>0$.  
\end{proof}

\begin{lem} Fix $\sigma>2$, $q_0(y)$ and $G(t,y)$ from Lemma~\ref{lem:dirichletsub}.  Let $\delta>0$.  There exists $\tau_\pm(t;\delta,\sigma,q_0)$ and a $T^*(\sigma,\delta,q_0)>0$ such that the  function $\underline{v}$ is a sub-solution for $y\in[\tau_-(t),\tau_+(t)]$ and $t>T^*$. 
\end{lem}
\begin{proof}
We will compare the function $\underline{v}$ to the sub-solution in Lemma~\ref{lem:dirichletsub}.  In particular, we seek those $x$ values for which $\underline{v}(t,x)\leq q(t,x-\sigma t)$.  We work in the moving coordinate system $y=x-\sigma t$ and seek solutions to the nonlinear equation,
\[ e^{(1-\frac{\sigma^2}{4})t}e^{-\frac{\sigma}{2}y}e^{-\frac{y^2}{4t}}G(t,y)= e^{\nu y}e^{-\delta t},\]
where we simplify notation and use $\nu$ to denote $\nu_v^-(\sigma)$.  This is equivalent to,
\[ -\frac{y^2}{4t}-\left(\frac{\sigma}{2}+\nu\right)y+(1-\frac{\sigma^2}{4}+\delta)t+\log \left( G(t,y)\right)=0.\]
Rescale $y=2tz$ and divide by $-t$ to find, 
\be z^2+\left(\sigma+2\nu\right)z-(1-\frac{\sigma^2}{4}+\delta)-\frac{1}{t}\log \left( \tilde{G}(t,z)\right)=0.\label{eq:implicit} \ee
Let $\rho=t^{-1/2}$, then the left hand side of (\ref{eq:implicit}) defines an implicit function 
\[ F(\rho,z):=  z^2+\left(\sigma+2\nu\right)z-(1-\frac{\sigma^2}{4}+\delta)-\rho^2\log \left(\rho H(\rho,z)\right),\]
with 
\[ H(\rho,z)=\frac{2A(\rho^{-2})}{\sqrt{2\pi}}\int_0^\infty e^{-\frac{(y'\rho)^2}{4}}\sinh(zy')q_0(y')dy'.\]
When $\rho=0$ we find solutions for 
\[ z_\pm=-\nu-\frac{\sigma}{2}\pm\frac{1}{2}\sqrt{(\sigma+2\nu)^2+4(1-\frac{\sigma^2}{4}+\delta)},\]
which after simplification using the identity $\nu^2+\sigma\nu+1=0$ becomes
\[  z_\pm=\frac{1}{2}\sqrt{\sigma^2-4}\pm\sqrt{\delta}.\]  
Thus, $F(0,z_\pm)=0$ and $F$ is $C^1$ with $F_\rho(0,z_\pm)=0$ and $F_z(0,z_\pm)=z_\pm-z_\mp$.  The implicit function theorem implies the existence of a $\rho^*>0$ such that $z(\rho)=z_\pm+R_\pm(\rho)$ solves $F(\rho,z(\rho))=0$ for $\rho<\rho^*$.  Reverting to the $(t,y)$ variables, we have
\be \tau_\pm(t)=\sqrt{\sigma^2-4}t\pm2\sqrt{\delta}t+2R_\pm(t^{-1/2})t,\label{eq:tau}\ee
for all $t>T^*(\sigma,\delta,q_0)=\left(\frac{1}{\rho^*(\sigma,\delta,q_0)}\right)^2$.
\end{proof}

\begin{remark}
Note that the leading order term in (\ref{eq:tau}) is precisely the group velocity associated to the spatial mode $e^{\nu y}$. This can be interpreted as saying that the region in space where the solution of the Fisher-KPP equation resembles an exponential function with certain strong decay rate is an interval that propagates at the group velocity of that mode.  This observation was previously noted in \cite{booty93}.
\end{remark}

\subsection{A sub-solution for the $u$ component}\label{sec:subU}
We now turn our attention to the $u$ component and construct a compactly supported sub-solution for the $u$ component.  We do this for all wavespeeds, 
\[  \max\{2,2\sqrt{d\alpha}\}<s<s_{anom}.\]
We take the dynamics of the $v$ component to be fixed in the sense that initial data has been selected and Lemma~\ref{lem:dirichletsub} has been applied to yield a sub-solution for the $v$ component.  We will show that there exists a one-parameter family of functions $\underline{u}(t,x)$ such that 
\[ N(\underline{u})=\underline{u}_t-d\underline{u}_{xx}-\alpha (\underline{u}-\underline{u}^2) -\beta v(t,x)<0.\]
In the following section, we will show that this parameter can be selected so that $\underline{u}(t,x)<u(t,x)$ for some value of $t>0$.  
 
The sub-solution consists of a nonlinear traveling front solution with weak decay propagating with speed $s$, concatenated with a solution of the linearized dynamics of the $u$ component about the zero state.  These two sub-solutions are glued together at a unique point in the frame of reference moving with speed $\sigma$ with $s<\sigma<s_{anom}$.  That is, we consider
\be \underline{u}(t,x)=\left\{ \begin{array}{ccc} U_r(x-st) & x<\sigma t \\
  U_r((\sigma-s)t)\psi(x-\sigma t,t) & x\geq \sigma t \\
  0 & x\geq \sigma t+\Theta_+(t),\end{array}\right. \label{eq:usubsol} \ee
for some $\Theta_+(t)>0$.  Here $U_r(\cdot)$ is a nonlinear traveling front solution and $\psi(\cdot,t)$ is a solution of the linearized problem near zero. Continuity is enforced by requiring that $\psi(0,t)=1$.  There exists a one-parameter family of such functions corresponding to different translates of the weakly decaying front.  We let $r$ parameterize this family through the identity,
\[ U_r(r)=\frac{1}{2}.\]

\paragraph{Weakly decaying nonlinear fronts} 
The nonlinear traveling front $U_r(x-st)$ is a solution of the second order ordinary differential equation 
\[ dU_r''+sU_r'+\alpha(U_r-U_r^2)=0.\]
The existence of such traveling front solutions for the Fisher-KPP equation is well known, see for example \cite{aronson78}.  In the sub-critical regime where $s<2\sqrt{d\alpha}$, the fixed point at the origin has complex conjugate eigenvalues and the decay of the nonlinear front is oscillatory.  In the super-critical regime where $s>2\sqrt{d\alpha}$, the origin has two real eigenvalues and one can show by phase plane analysis that there exists a unique nonlinear front solution that approaches the origin with weak exponential decay.  This decay rate is prescribed by the dispersion relation and is $\nu_u^+(s,0)$.  The intermediate speed $s=2\sqrt{d\alpha}$ is the critical or minimal speed and the traveling front moving with this speed (referred to as the critical or Fisher-KPP front)  is the selected front for the $v$ component in isolation.

\paragraph{The function $\psi$.} In order to specify the functional form of $\psi$, we consider the linearized equation for the $u$ component with the sub-solution (\ref{eq:vsubsol}) in place of $v$.  That is, we consider,
\[ u_t=du_{yy}+\sigma u_y +\alpha u+\beta  e^{\nu_v^-(\sigma)y}e^{-\delta t} .\]
We consider a solution consisting of stationary exponential profiles and a particular solution describing the influence of the inhomogeneous term $\underline{v}$,
\[ u(t,x)=c_1 e^{\nu_u^+(\sigma)y}+c_2 e^{\nu_u^-(\sigma)y}+p(t,x).\]
Since we require weak exponential decay at the matching point, we set $c_2=0$.  To determine the particular solution, we apply the  transformation $w(t,x)=e^{\delta t}u(t,x)$ that removes the non-autonomous term from the inhomogeneity, leading to an equation for $w$, 
\[ w_t=dw_{yy}+\sigma w_{y}+(\alpha+\delta)w+\beta e^{\nu_v^-(\sigma)y}.\]
Finding the particular solution here and reverting to the original coordinates we obtain
\[ p(t,x)=-\frac{\beta}{d(\nu_v^-(\sigma))^2+\sigma\nu_v^-(\sigma)+\alpha+\delta} e^{\nu_v^-(\sigma)y}e^{-\delta t}.\]
For brevity going forward we introduce the notation,
\[ D(\nu)=d\nu^2+\sigma\nu+\alpha+\delta.\]
Note that $D(\nu_v^-(\sigma))>0$.  We are now in a position to select,
\be \psi(x-\sigma t,t)=c_1(t) e^{\nu_u^+(\sigma)(x-\sigma t)}- \frac{ \beta}{D(\nu_v^-(\sigma))} e^{\nu_v^-(\sigma)(x-\sigma t)}e^{-\delta t}.\label{eq:psi} \ee
In order for $\underline{u}(t,x)$ to be continuous at $x-\sigma t=0$ we require $\psi(0,t)=1$, which imposes 
\[ c_1(t)= \left(1+\frac{ \beta}{D(\nu_v^-(\sigma))} e^{-\delta t}\right).\]
Note that $\nu_u^+(\sigma)<\nu_v^-(\sigma)<0$ so that $\psi(y,t)$ is negative for large $y$.  In fact, the solution $\psi(x-\sigma t,t)$ vanishes at the point $y=\Theta_+(t)$ with
\[ \Theta_+(t)= \frac{1}{\nu_v^-(\sigma)-\nu_u^+(\sigma)} \log
 \left( \frac{c_1(t)D(\nu_v^-(\sigma))}{ \beta} e^{\delta t}\right),\]
which after some simplification is equivalent to
\be \Theta_+(t)=\frac{\delta }{\nu_v^-(\sigma)-\nu_u^+(\sigma)}t+\frac{1}{\nu_v^-(\sigma)-\nu_u^+(\sigma)} \log
 \left(\frac{D(\nu_v^-(\sigma))}{ \beta}+e^{-\delta t} \right).\label{eq:Thetaplus} \ee

\paragraph{Selecting $\delta$ and a decomposition of the real line}
We have shown that for any $\delta>0$, the function $\underline{v}$ is a sub-solution on the interval $[\tau_-(t),\tau_+(t)]$.  We have also  constructed for any $\delta>0$ a candidate sub-solution for the $u$ component with the properties
\[ \psi(0,t)=1, \quad \psi(\Theta_+(t),t)=0, \quad \psi(y,t)>0 \ \text{for all} \ 0\leq y<\Theta_+(t).\]  
We now will restrict to a particular choice of $\delta$.  We set 
\be \delta_c=\sqrt{\sigma^2-4}\left(\nu_v^-(\sigma)-\nu_u^+(\sigma)\right).\label{eq:delta}\ee
\begin{lem}\label{lem:Td} Let $\delta=\delta_c$.  Then there exists a $T_\delta(\sigma,q_0)>0$ such that 
\be \tau_-(t)<\Theta_+(t)<\tau_+(t), \quad \text{for all} \ \ t>T_\delta(\sigma,q_0).\label{eq:tautheta} \ee
 holds for all $t>T_\delta$. Moreover, if $\sigma$ is large enough then $\tau_-(t)>0$ for $t>T_\delta$. 
\end{lem}

\begin{proof} Recall that $\tau_\pm$ are defined for all  $t>T^*(\sigma,\delta_c,q_0)$.  Compare formulas  (\ref{eq:tau}) and (\ref{eq:Thetaplus}).  We note that $R_\pm(t)\to 0$ as $t\to\infty$ and that the correction term in (\ref{eq:Thetaplus}) is bounded and approaches a finite limit as $t\to\infty$.  This implies that there exists a $T_\delta(\sigma,q_0)\geq T^*(\sigma,\delta_c,q_0)$ such that (\ref{eq:tautheta}) holds.   Note that as $\sigma\to s_{anom}$, $\delta_c$ tends to zero and $\tau_-(t)>0$ for $t$ sufficiently large.  
\end{proof}

Still working in moving coordinate frame $y=x-\sigma t$, we decompose the real line into into four regions as follows,
\begin{eqnarray*}
I_a&=& (-\infty,0] \\
I_b&=& (0,\tau_-(t)] \\
I_c&=& (\tau_-(t),\Theta_+(t)] \\
I_d&=& (\Theta_+(t),\infty) 
\end{eqnarray*}
This decomposition holds for all $t>T_\delta(\sigma,q_0)$.  In region $I_a$, the candidate sub-solution $\underline{u}$ is given by the weakly nonlinear front $U_{r}$.  In region $I_d$, $\underline{u}$ is zero.  We emphasize the difference between $I_b$ and $I_c$.  On $I_c$ we have that the function $\underline{v}$ is a sub-solution for the $v$ component. This fact will aid in the proof that $\underline{u}$ is a sub-solution under certain restrictions on the parameters.  However, in region $I_b$, the function  $\underline{v}$ is not generally a sub-solution.  Nonetheless, in this region we can exploit the fact that $\psi(x-\sigma t,t)$ is bounded above zero to show that $\underline{u}$ remains a sub-solution in this region.

\paragraph{The sub-solution}

\begin{lem}\label{lem:usub} Recall that $v(t,x)>0$ is a fixed function of the initial data $v_0(x)$ and that there exists a $q_0(x)$ such that the conclusions of Lemmas~\ref{lem:dirichletsub}-\ref{lem:Td} hold for some fixed values of $s$ and $\sigma$ satisfying $\max\{ 2, 2\sqrt{d\alpha}\}<s<\sigma<s_{anom}$ and for  $\delta=\delta_c$.  Then there exists $T_u(s,\sigma,q_0)\geq \max\{T^*,T_\delta\}$ and a $r_c(T_u,s,\sigma,q_0)$ such that for all $r<r_c$, the function $\underline{u}(t,x)$ as defined in (\ref{eq:usubsol}) is a sub-solution for the $u$ component for all $t>T_u$.  
\end{lem}
\begin{proof}
We need to show that $N(\underline{u})\leq 0$ for all $x$ and $t>T_u$, where
\[ N(u)=u_t-du_{xx}-\alpha u+\alpha u^2 -\beta v(t,x) .\]
We proceed by straightforward calculation.  In region $I_a$, we have
\[ N(U_r(x-st))= -\beta v(t,x)<0.\]
In region $I_d$, the calculation is similarly simple since we have $\underline{u}=0$ and
\[ N(0)=-\beta v(t,x)<0.\]
In the intermediate regimes in region $I_b$ and $I_c$ the analysis is more involved.   In regions $I_b$ and $I_c$, $N(\underline{u})$ is functionally equivalent.  We differentiate between the two regimes because we will need to argue along different lines to show that $\underline{u}$ is a sub-solution in each case.  In both regions we have,
\begin{eqnarray}
N(\underline{u})&=&(\sigma-s)U_r'(\cdot)\psi+U_r(\cdot)c_1'(t)e^{\nu_u^+(\sigma)(x-\sigma t)} \nonumber \\
&+& U_r(\cdot) \beta e^{\nu_v^-(\sigma)(x-\sigma t)}e^{-\delta_c t}-\beta v(t,x)+\alpha (U_r\psi)^2 \label{eq:Nu}
\end{eqnarray}
First, note that $U_r(\cdot)$ is the solution of a second order ordinary differential equation, $dU_r''+sU_r'+\alpha (U_r-U_r^2)$.  Phase plane analysis implies that the $U_r'$ can be written as a function of $U_r$.  Since $s>2\sqrt{d\alpha}$ we see that as $U_r$ tends to zero we have the expansion,
\be U_r'=\nu_u^+(s)U_r(1+R(U_r)),\label{eq:Usmall}\ee
where $|R(U_r)|<CU_r$ for $U_r$ sufficiently small.

Consider region $I_c$. Note that $c_1'(t)<0$.  In region $I_c$, we have that $\underline{v}<v(t,x)$ for all $t>T^*$.  Therefore, we can regroup,
\begin{eqnarray}
N(\underline{u})&=& \left( (\sigma-s)\nu_u^+(s)+R(U_r(\cdot))+\alpha U_r(\cdot)\psi\right)U_r(\cdot)\psi \nonumber \\
&+&U_r(\cdot)c_1'(t)e^{\nu_u^+(\sigma)(x-\sigma t)}  -\beta (v(t,x)-U_r(\cdot)\underline{v}(t,x)).
\label{eq:Nsimple} \end{eqnarray}
Note that $(\sigma-s)\nu_u^+(s)$ is a fixed negative number.  Since $U_r\to 0$ as $\tau\to -\infty$, there exists a $r_0(s,\sigma,q_0)$ such that for all $r<r_0(s,\sigma,q_0)$ and $t>T_\delta(\sigma,q_0)$ we have $(\sigma-s)\nu_u^+(s)+R(U_r)+\alpha U_r\psi<0$ for all $x\in I_c$.  We have thus shown that $N(\underline{u})<0$ in region $I_c$.  

We now turn our attention to region $I_b$.  Here we do not necessarily have that the function $\underline{v}$ is a sub-solution for the $v$ component.  Nonetheless,  consider (\ref{eq:Nu}).  We simplify as in (\ref{eq:Nsimple}),
\begin{eqnarray*} N(\underline{u})&=&\left( (\sigma-s)\nu_u^+(s)+R(U_r)+\alpha U_r\psi\right)U_r\psi \\
&+&U_r(\cdot)c_1'(t)e^{\nu_u^+(\sigma)(x-\sigma t)}+U_r(\cdot)\beta e^{\nu_v^-(\sigma)(x-\sigma t)}e^{-\delta_c t}  -\beta v(t,x).
\end{eqnarray*}
We have shown above that the first term can be made negative and bounded away from zero by selecting $r$ appropriately.  Since $c_1'(t)<0$, we need only to control for the positive term,
\[ U_r(\cdot)\beta e^{\nu_v^-(\sigma)(x-\sigma t)}e^{-\delta_c t}.\]
To do so, we will divide and multiply by $\psi(x-\sigma t,t)$ and show that 
\be \frac{\beta e^{\nu_v^-(\sigma)(x-\sigma t)}e^{-\delta_c t}}{\psi(x-\sigma t,t)}=\frac{\beta }{c_1(t) e^{(\nu_u^+(\sigma)-\nu_v^-(\sigma))(x-\sigma t)}e^{\delta_c t}- \frac{ \beta}{D(\nu_v^-(\sigma))} },\label{eq:keyterm}\ee
can be made arbitrarily small for $(t,x)\in I_b$ by restricting to $t$ sufficiently large if and only if
\[ (\nu_u^+(\sigma)-\nu_v^-(\sigma))(x-\sigma t)+\delta_c t>0,\] 
for all $(t,x)\in I_b$. At the left boundary of $I_b$, i.e. when $x-\sigma t=0$, positivity follows from the positivity of $\delta_c$.  At the right bounday, we substitute $x-\sigma t=\tau_-(t)$ and find 
\begin{eqnarray*} (\nu_u^+(\sigma)-\nu_v^-(\sigma))\tau_-(t)+\delta_c t&=& (\nu_u^+(\sigma)-\nu_v^-(\sigma)) \left(-2\nu_v^-(\sigma)t-2\sqrt{\delta_c}t+2R_-(t)t\right) \\ &+& 2\nu_v^-(\sigma)(\nu_u^+(\sigma)-\nu_v^-(\sigma)) t \\
&=& t(\nu_u^+(\sigma)-\nu_v^-(\sigma))\left( -2\sqrt{\delta_c}+2R_-(t)\right) 
\end{eqnarray*}
Recall that $R_-(t)\to 0$ as $t\to \infty$.  Since $\nu_u^+(\sigma)-\nu_v^-(\sigma)<0$ we find that the exponent is positive and (\ref{eq:keyterm}) can be made arbitrarily small by restricting to $t$ sufficiently large.  This implies that $N(\underline{u})$ can be factored as 
\[ N(\underline{u})=\left( (\sigma-s)\nu_u^+(s)+R(U_r)+\alpha U_r\psi +\frac{\underline{v}}{\psi}\right)U_r\psi+U_r(\cdot)c_1'(t)e^{\nu_u^+(\sigma)(x-\sigma t)}  -\beta v(t,x).
\]
Since $(\sigma-s)\nu_u^+(s)+R(U_r)+\alpha U_r\psi$ is negative and bounded away from zero and since $\underline{v}/\psi$ can be made uniformly arbitrarily small, there exists a $T_u(s,\sigma,q_0)\geq T_\delta(\sigma,q_0)$ such that $\underline{u}$ is a sub-solution in region $I_b$.  

To complete the proof of the lemma, we need to verify that $\underline{u}$ is a sub-solution at the matching point $x-\sigma t =0$.  We must show that the left $x$ derivative of $\underline{u}$ at the right boundary of region $I_a$ is steeper than the corresponding right derivative at the left boundary of $I_b$.  That is, we require
\be \lim_{y\to 0^-} \frac{\partial \underline{u}}{\partial y}<\lim_{y\to 0^+} \frac{\partial \underline{u}}{\partial y}<0.\label{eq:Dcond} \ee

The spatial  derivative of $\underline{u}$ in region $I_a$ is $U_r'(x-st)$.  Making use of the expansion (\ref{eq:Usmall}), the limit of the derivative from the left is
\[ \lim_{y\to 0^-} \frac{\partial \underline{u}}{\partial y}= \nu_u^+(s)U_r((\sigma-s)t)(1+R(U_r)).\]
We compare this to the the spatial derivative in region $I_b$ at the matching point $x-\sigma t=0$.  This derivative is $U_r((\sigma-s)t)\partial_x \psi(x-st,t)$.  At the matching point the limit from the right is,
\[ \lim_{y\to 0^+} \frac{\partial \underline{u}}{\partial y}=U_r((\sigma-s)t)\left(\nu_u^+(\sigma)c_1(t) - \frac{\nu_v^-(\sigma) \beta}{D(\nu_v^-(\sigma))} e^{-\delta_c t}\right). \]
After substituting the expression for $c_1(t)$, the derivative becomes
\[ U_r((\sigma-s)t)\left(\nu_u^+(\sigma)+ \frac{\left(\nu_u^+(\sigma)-\nu_v^-(\sigma)\right) \beta}{D(\nu_v^-(\sigma))} e^{-\delta_c t}\right).\]
Note that $\nu_u^+(\sigma)<\nu_v^-(\sigma)$.  Since $s<\sigma$, we also have $\nu_u^+(s)<\nu_u^+(\sigma)<0$.
Selecting $T_u$ perhaps even larger we can enforce,
\[ \nu_u^+(s)<\nu_u^+(\sigma)+ \frac{\left(\nu_u^+(\sigma)-\nu_v^-(\sigma)\right) \beta}{D(\nu_v^-(\sigma))} e^{-\delta_c t},\]
from which we may select $r_c(T_u,s,\sigma,q_0)\leq r_0(s,\sigma,q_0)$ so that the
the derivative condition (\ref{eq:Dcond}) is satisfied.

\end{proof}

\subsection{Super-solutions}
The sub-solutions constructed in the previous section will allow us to bound the spreading speed of (\ref{eq:main}) from below.  To demonstrate that the selected spreading speed is actually the anomalous one we will need super solutions that bound the spreading speed from above.  In the following lemma, super-solutions are constructed from solutions of the linearized equations.  

\begin{lem}~\label{lem:supersol} Fix $v_0(x)$ and therefore the solution to the initial value problem, $v(t,x)$. Let $s>s_{anom}$. There exists  $C_v>0$ and $C_u^*(C_v)$ such that for all $C_u>C_u^*$ there exists a $\theta(C_u,C_v)$ for which
\[ \bar{u}(t,x)=\left\{ \begin{array}{cc} C_ue^{\nu_u^+(s)(x-st)}+C_v\kappa e^{\nu_v^-(s)(x-st)} & x-st\geq \theta  \\
\frac{1}{2}+\frac{1}{2}\sqrt{1+\frac{4\beta}{\alpha}} & x-st<\theta \end{array}\right. ,\]
with
\[ \kappa=-\frac{\beta}{d_u(\nu_v^-(s),0)},\]
is a super-solution for the $u$ component.  
\end{lem}

\begin{pf}
A similar result was established in \cite{holzer-anomalous}, we adapt the arguments to the current situation.  First, note that $\bar{v}(t,x)=\min\{1,C_ve^{\nu_v^-(s)(x-st)}\}$ is a super-solution for the $v$ component and given $v_0(x)$ we can select $C_v$ so that $\bar{v}(t,x)\geq v(t,x)$ for all $t$.  With $C_v$ fixed, we find that these two curves intersect at the point $y_v=\frac{1}{\nu_v^-(s)}\log\frac{1}{C_v}$ in the moving frame $y=x-st$.  

Observe that when $s>s_{anom}$ we have $\nu_v^-(s)<\nu_u^+(s)$ and $d_u(\nu_v^-(s),0)<0$.  This implies that $\kappa>0$ and the sum of exponentials in the definition of $\bar{u}(t,x)$ is positive.  

Now, given $C_v$ we can find a $C_u^*(C_v)$ such that for any $C_u>C_u^*$, the intersection point $\theta(C_u,C_v)$ is greater than $y_v$.  Factor $N(u)$ as follows,
\[ N(u)=u_t-du_{xx}-\alpha u+\alpha u^2+\beta(\bar{v}(t,x)-v(t,x))-\beta \bar{v}(t,x).\]

We argue in pieces.  First, for $x-st<\theta$ we have that the constant function $1$ is a super-solution for the $v$ component.  Then evaluate $N(u)$ at a constant value $u_c$, 
\[ N(u_c)=-\alpha u_c+\alpha u_c^2-\beta +\beta(1-v(t,x)).\]
Taking $u_c=\frac{1}{2}+\frac{1}{2}\sqrt{1+\frac{4\beta}{\alpha}}$ to be a root of the polynomial  $-\alpha u+\alpha u^2-\beta$ and noting $v(t,x)<1$ we have that $N(-\alpha u_c+\alpha u_c^2-\beta )>0$.  

We now consider $x-st>\theta$.  Here the functional form of $\bar{u}(t,x)$ is the sum of exponentials given above.  The exponential terms are chosen precisely so that the linear terms vanish and all that is left is,
\[ N(\bar{u}(t,x)=\alpha u^2+\beta(\bar{v}(t,x)-v(t,x)),\]
which is clearly positive.  This completes the proof.

\end{pf}

\section{Spreading speeds from Heaviside step function initial data -- the proof of Theorem~\ref{thm:main}}\label{sec:proof}
We now prove Theorem~\ref{thm:main}.

\begin{proof}
Consider any $s$ satisfying $\max\{2,2\sqrt{d\alpha}\}<s<s_{anom}$.  Fix $s<\sigma<s_{anom}$ and select $\delta=\delta_c$ as in Lemma~\ref{lem:Td}.  Recall also that we require $\sigma$ large enough so that $\tau_-(t)>0$.  Without loss of generality we may assume that the region where $v(0,x)>0$ intersects the positive half line.  Consider any compactly supported function $q_0(x)$, not identically zero, whose support lies in $\mathbb{R}^+$ such that $q_0(x)\leq v(0,x)$ for all $x>0$.  Then Lemma~\ref{lem:dirichletsub} gives the existence of a sub-solution for the $v$ component that is supported on the positive half line with the properties listed in Lemma~\ref{lem:dirichletsub}.  

Now evolve the solution further until $t=T_u(s,\sigma,q_0)$.  Lemma~\ref{lem:usub} now gives the existence of a one parameter family of sub-solutions depending on $s,\sigma$ and $q_0(x)$ parameterized by $r\leq r_c(T_u,s,\sigma,q_0)$.  The maximum principle implies that $u(T_u,x)>0$ for all $x$.   We now claim that we can select a particular $r\in\mathbb{R}$ with $r\leq r_c$ so that the sub-solution (\ref{eq:usubsol}) satisfies $\underline{u}(T_u,x)<u(T_u,x)$.   

We do this in two parts.  We will require sub-solutions for the $u$ component in isolation.  These can be constructed from subcritical nonlinear traveling fronts, see for example \cite{aronson78}.  Let $u_{osc}(x-(2\sqrt{d\alpha}-\e)t)$ be one such nonlinear traveling front solution, cut off and set equal to zero for all $x$ to greater than its smallest zero.  We can select a particular translate of this front so that $u(t,x)>u_{osc}(t,x)$ for all $0<t<T_u(\sigma.q_0)$.  At $t=T_u$, if the support of $u_{osc}(t,x)$ intersects $I_b$, then we claim that there exists a $r_1\leq r_c(s,\sigma,q_0)$ such that for all $r<r_1$, we have $\underline{u}(T_u,x)<u_{osc}(T_u,x)<u(T_u,x)$ for all points in $I_a$.  This follows from the decay rates at $-\infty$.  The unstable eigenvalue of the linearization of the traveling wave equation $dU''+sU'+\alpha(U-U^2)=0$ at $u=1$ has a single unstable eigenvalue 
\[ e(s)=-\frac{s}{2d}+\frac{1}{2d}\sqrt{s^2+4d\alpha},\]
from which we observe that $e$ is a decreasing function of $s$.  This implies that the front $u_{osc}$ converges to $u=1$ at a faster rate than $U_r$ and for fixed $t$, $U_r<u_{osc}$ on some half interval near $x=-\infty$.  If at $t=T_u$ the support of $u_{osc}$ includes points in $I_b$, then we can find a $r\leq r_c$ so that $U_r<u_{osc}$ for all points in $I_b$.  

If the support of $u_{osc}(t,x)$ does not intersect $I_b$, we can nonetheless find a translate of $U_r$ such that $\underline{u}<u(T_u,x)$. This follows since for $t=T_u$ the region where $u_{osc}$ is equal to zero intersected with the region $I_a$ is compact.  The solution $u(t,x)>0$ here and bounded from below so by selecting $r$ smaller if necessary we find $r$ such that $U_r(x-sT_u)<u(T_u,x)$ for all points in $I_a$.  

A similar arguments hold for $y\in[0,\Theta_+(T_u)]$.  Note that $\Theta_+$ is independent of $r$.  Again, $u(t,x)$ is bounded from below on this interval and $\underline{u}(T_u,x)$ is bounded above by $U_r((\sigma-s)T_u)$ which implies that we can select $r$ again smaller so that $\underline{u}(T_u,x)<u(T_u,x)$ for all $x$.

Recall the definition of the invasion point in Theorem~\ref{thm:main}.  For the sub-solution $\underline{u}(t,x)$, the invasion point is given explicitly as $\kappa(t)=st+r$ and since $\underline{u}(t,x)<u(t,x)$ for all $x$ and all $t>T_0$, we have that $s_{sel}\geq s$.  This construction can be performed for all $\max\{2,2\sqrt{d\alpha}\}<s<\sigma<s_{anom}$ and we therefore find that $s_{sel}\geq s_{anom}$.

Conversely, let $s>s_{anom}$ and consider super-solutions of the form given in Lemma~\ref{lem:supersol}.  For fixed $v_0(x)$, we can find a $C_v$ such that $\bar{v}(t,x)$ is a super-solution for the $v$ component.  Since $u_0(x)$ is also compactly supported we can find $C_u>0$ and sufficiently large so that $\bar{u}(t,x)\geq u_0(x)$ and therefore $\bar{u}(t,x)\geq u(t,x)$ for all $t$.  This implies that $s_{sel}\leq s$ for any $s>s_{anom}$.  

We have therefore shown that $s_{anom}\leq s_{sel}\leq s_{anom}$ and therefore $s_{sel}=s_{anom}$. 

\end{proof}

\section{Discussion}\label{sec:discussion}

We have shown that anomalous spreading speeds are observed in the nonlinear system (\ref{eq:main}) for those parameter values in the {\em relevant} regime, i.e. for $(d,\alpha)\in\mathrm{IV}$.  We show that Heaviside step function initial data propagates with the anomalous speed.  We now conclude with some observations and discussion of generalizations of the current work and future directions for research.

The methods utilized here generalize directly to some other coupled reaction-diffusion equations.  One example is the following system comprised of a heat equation coupled to a Fisher-KPP equation,
\begin{eqnarray*}
u_t &=& du_{xx}+\alpha(u-u^2)+\beta v \\
v_t &=& v_{xx} -\gamma v,
\end{eqnarray*}
where $\gamma>0$ enforces pointwise exponential decay of the $v$ component.   This example was introduced in \cite{holzer-anomalous}, we refer the reader there for more details. The temptation is to expect that the selected spreading speed for the $u$ component is the Fisher-KPP speed $2\sqrt{d\alpha}$.  As is the case with (\ref{eq:main}), there exists a subset of parameters for which anomalous spreading speeds exist for the linear system.  The proof of Theorem~\ref{thm:main} is readily adapted to prove that the selected spreading speed for the nonlinear system is also the anomalous speed.  

It would also be of interest to know whether a general result concerning cooperative systems like the ones in \cite{weinberger02,li05} could be adapted to the anomalous case.  Consider a general system of coupled reaction-diffusion equation whose linearization is block triangular.  One major obstacle is that as the number of equations is increased the number of possible pinched double root combinations also increases.  Thus, it would be challenging to construct a general theory that accounts for which of these pinched double roots are relevant or irrelevant.  

Also of interest would be generalizations to larger classes of nonlinearities.  Simple modifications would be amenable to the analysis presented here; for example $u(1-u)$ could easily be replaced with $f(u)$ where $f$ is of KPP type.  The inhomogeneous term $\beta v$ could be replaced with a term $\beta v g(u,v)$ with $g(u,v)>0$ for $u,v>0$ and $|g(u,v)|=1+\O(|u|+|v|)$.  Further generalizations where the $v$ component is allowed to depend on the $u$ component would be more challenging.  In a different direction there are also large classes of equations for which the selected spreading speed is a nonlinear speed.  It would be interesting to quantify the role that anomalous linear speeds could play in this context.  

Finally, we mention that the proof of Theorem~\ref{thm:main} requires the existence of a comparison principle, although the phenomena is observed in non-cooperative systems.  Generalizations to systems without comparison would require new techniques.  Some insights can be gleaned from a linear analysis.  We again mention section 8.3 of \cite{holzer-criteria} for a discussion of the differences between relevant and irrelevant double roots as well as implications to nonlinear phenomena.  For example, supposing the existence of a traveling front solution propagating slower than the linear spreading speed it is shown that relevant double roots enforce the existence of an unstable resonance pole.  Building a more complete theory is made more challenging by the fact that anomalous speeds inherently involve different components spreading at different speeds and interacting over larger spatial scales where a standard traveling wave analysis may not be applicable.

\bibliographystyle{abbrv}
\bibliography{MHmaster}

\def\cprime{$'$}
\begin{thebibliography}{10}

\bibitem{aronson78}
D.~G. Aronson and H.~F. Weinberger.
\newblock Multidimensional nonlinear diffusion arising in population genetics.
\newblock {\em Adv. in Math.}, 30(1):33--76, 1978.

\bibitem{bell09}
S.~Bell, A.~White, J.~Sherratt, and M.~Boots.
\newblock Invading with biological weapons: the role of shared disease in
  ecological invasion.
\newblock {\em Theoretical Ecology}, 2:53--66, 2009.
\newblock 10.1007/s12080-008-0029-x.

\bibitem{booty93}
M.~R. Booty, R.~Haberman, and A.~A. Minzoni.
\newblock The accommodation of traveling waves of {F}isher's type to the
  dynamics of the leading tail.
\newblock {\em SIAM J. Appl. Math.}, 53(4):1009--1025, 1993.

\bibitem{bramson83}
M.~Bramson.
\newblock Convergence of solutions of the {K}olmogorov equation to travelling
  waves.
\newblock {\em Mem. Amer. Math. Soc.}, 44(285):iv+190, 1983.

\bibitem{elliott12}
E.~C. Elliott and S.~J. Cornell.
\newblock Dispersal polymorphism and the speed of biological invasions.
\newblock {\em PLoS ONE}, 7(7):e40496, 07 2012.

\bibitem{fife77}
P.~C. Fife and J.~B. McLeod.
\newblock The approach of solutions of nonlinear diffusion equations to
  travelling front solutions.
\newblock {\em Arch. Ration. Mech. Anal.}, 65(4):335--361, 1977.

\bibitem{fisher37}
R.~A. Fisher.
\newblock The wave of advance of advantageous genes.
\newblock {\em Annals of Human Genetics}, 7(4):355--369, 1937.

\bibitem{freidlin91}
M.~Freidlin.
\newblock Coupled reaction-diffusion equations.
\newblock {\em Ann. Probab.}, 19(1):29--57, 1991.

\bibitem{ghazaryan10}
A.~Ghazaryan, P.~Gordon, and A.~Virodov.
\newblock Stability of fronts and transient behaviour in {KPP} systems.
\newblock {\em Proc. R. Soc. Lond. Ser. A Math. Phys. Eng. Sci.},
  466(2118):1769--1788, 2010.

\bibitem{hamel13}
F.~Hamel, J.~Nolen, J.-M. Roquejoffre, and L.~Ryzhik.
\newblock A short proof of the logarithmic {B}ramson correction in
  {F}isher-{KPP} equations.
\newblock {\em Netw. Heterog. Media}, 8(1):275--289, 2013.

\bibitem{holzer-anomalous}
M.~Holzer.
\newblock Anomalous spreading in a system of coupled {F}isher-{KPP} equations.
\newblock {\em Phys. D}, 270:1--10, 2014.

\bibitem{holzerLV}
M.~Holzer and A.~Scheel.
\newblock A slow pushed front in a {L}otka {V}olterra competition model.
\newblock {\em Nonlinearity}, 25(7):2151, 2012.

\bibitem{holzer-criteria}
M.~Holzer and A.~Scheel.
\newblock Criteria for pointwise growth and their role in invasion processes.
\newblock {\em J. Nonlinear Sci.}, 24(4):661--709, 2014.

\bibitem{iida11}
M.~Iida, R.~Lui, and H.~Ninomiya.
\newblock Stacked fronts for cooperative systems with equal diffusion
  coefficients.
\newblock {\em SIAM Journal on Mathematical Analysis}, 43(3):1369--1389, 2011.

\bibitem{kolmogorov37}
A.~Kolmogorov, I.~Petrovskii, and N.~Piscounov.
\newblock Etude de l'equation de la diffusion avec croissance de la quantite'
  de matiere et son application a un probleme biologique.
\newblock {\em Moscow Univ. Math. Bull.}, 1:1--25, 1937.

\bibitem{li05}
B.~Li, H.~F. Weinberger, and M.~A. Lewis.
\newblock Spreading speeds as slowest wave speeds for cooperative systems.
\newblock {\em Math. Biosci.}, 196(1):82--98, 2005.

\bibitem{raugel98}
G.~Raugel and K.~Kirchg{\"a}ssner.
\newblock Stability of fronts for a {KPP}-system. {II}. {T}he critical case.
\newblock {\em J. Differential Equations}, 146(2):399--456, 1998.

\bibitem{vansaarloos03}
W.~van Saarloos.
\newblock Front propagation into unstable states.
\newblock {\em Physics Reports}, 386(2-6):29 -- 222, 2003.

\bibitem{weinberger02}
H.~F. Weinberger, M.~A. Lewis, and B.~Li.
\newblock Analysis of linear determinacy for spread in cooperative models.
\newblock {\em J. Math. Biol.}, 45(3):183--218, 2002.

\bibitem{weinberger07}
H.~F. Weinberger, M.~A. Lewis, and B.~Li.
\newblock Anomalous spreading speeds of cooperative recursion systems.
\newblock {\em J. Math. Biol.}, 55(2):207--222, 2007.

\end{thebibliography}

\end{document}